\documentclass{amsart}
\usepackage{amsfonts, color}
\usepackage{latexsym}
\usepackage{amssymb}
\usepackage{amsmath}

\newcommand{\R}{\mathbb R}

\newcommand{\E}{\mathbb E}

\newtheorem{thm}{Theorem}[section]
\newtheorem{cor}{Corollary}[section]
\newtheorem{lemma}{Lemma}[section]

\theoremstyle{remark}

\begin{document}


\title{The variance conjecture on projections of the cube}

\author[D.\,Alonso]{David Alonso-Guti\'errez}
\address{\'Area de an\'alisis matem\'atico, Departamento de matem\'aticas, Facultad de Ciencias, Universidad de Zaragoza, Pedro cerbuna 12, 50009 Zaragoza (Spain), IUMA}
\email[(David Alonso)]{alonsod@unizar.es}


\author[J.\,Bernu\'es]{Julio Bernu\'es}
\address{\'Area de an\'alisis matem\'atico, Departamento de matem\'aticas, Facultad de Ciencias, Universidad de Zaragoza, Pedro cerbuna 12, 50009 Zaragoza (Spain), IUMA}
\email[(Julio Bernu\'es)]{bernues@unizar.es}
\thanks{Partially supported by MINECO Spanish grant MTM2016-77710-P and DGA grant E-64}

\begin{abstract}
We prove that the uniform probability measure $\mu$ on every $(n-k)$-dimensional projection of the $n$-dimensional unit cube verifies the variance conjecture with an absolute constant $C$
$$
\textrm{Var}_\mu|x|^2\leq  C \sup_{\theta\in S^{n-1}}\E_\mu\langle x,\theta\rangle^2\E_\mu|x|^2,
$$
provided that $1\leq k\leq\sqrt n$. We also prove that if $1\leq k\leq n^{\frac{2}{3}}(\log n)^{-\frac{1}{3}}$, 
the conjecture is true for the family of uniform probabilities on its projections on random $(n-k)$-dimensional subspaces.
\end{abstract}


\date{\today}
\maketitle
\section{Introduction and notation}

The (generalized) variance conjecture states that there exists an absolute constant $C$ such that for every centered log-concave probability $\mu$ on $\R^n$  (i.e. of the form $d\mu=e^{-v(x)}dx$ for some convex function $v:\R^n\to(-\infty,\infty]$)
$$
\textrm{Var}_\mu|x|^2\leq C \lambda_\mu^2\E_\mu|x|^2,
$$
where $\E_\mu$ and $\textrm{Var}_\mu$ denote the expectation and the variance with respect to $\mu$ and $\lambda_\mu$ is the largest eigenvalue of the covariance matrix, i.e.  $\lambda_\mu^2=\max_{\theta\in S^{n-1}}\E_\mu\langle x,\theta\rangle^2$ where $S^{n-1}$ denotes the unit Euclidean sphere in $\R^n$.

This conjecture was first considered in the context of the so called Central Limit Problem for isotropic convex bodies in \cite{BK} and it is a particular case of a more general statement, known as the Kannan, Lov\'asz, and Simonovits or KLS-conjecture, see \cite{KLS}, which conjectures the existence of an absolute constant $C$ such that for any centered log-concave probability in $\R^n$ and any locally Lipschitz
function $g:\R^n\to\R$ such that  $\textrm{Var}_\mu g(x)$ is finite
$$
\textrm{Var}_\mu\,g(x)\leq C\lambda_\mu^2\E_\mu\,|\nabla g(x)|^2.
$$

In recent years a number of families of measures have been proved to verify these conjectures (see \cite{AB2} for a recent review on the subject). For instance, the KLS-conjecture is known to be true for the Gaussian probability and
the uniform probability measures on
the $\ell_p^n$-balls, some revolution bodies, the simplex and, with an extra
$\log n$ factor, on unconditional
bodies and log-concave probabilities with many symmetries (see \cite{BaC},
\cite{BaW}, \cite{B}, \cite{H}, \cite{K}, \cite{LW}, \cite{S}).
The best general known result for the KLS-conjecture adds a factor $\sqrt{n}$ and is due to Lee and Vempala (see \cite{LV}).
Besides, the variance conjecture is known to be true for uniform probabilities on unconditional bodies and on hyperplane projections of the cross-polytope and the cube (see \cite{K}
and \cite{AB1}). The best general estimate for the variance conjecture is the one given by Lee and Vempala for the KLS-conjecture.

We would like to remark that, while in the case of the KLS-conjecture one can assume without loss of generality that $\mu$ is isotropic (since then every linear transformation of the measure verifies it) this is not the case when we restrict to the variance conjecture, as we are considering only the function $g(x)=|x|^2$.

Before stating our results let us introduce some more notation. Let
$$
B_\infty^n:=\{x\in\R^n\,:\,|x_i|\leq1,\,\forall \,1\leq i\leq n\}
$$
denote the $n$-dimensional unit cube and, for any $1\le k\le n$, let $G_{n,n-k}$ be the set of all $(n-k)$-dimensional subspaces of $\R^n$. For any $E\in G_{n,n-k}$ we will denote by $K:=P_EB_\infty^n $ the orthogonal projection of $B_\infty^n$ onto $E$ and by $\mu$ the uniform probability on $K$. As mentioned before, it was proved in \cite{AB1} that the family of uniform probabilities on any $(n-1)$-dimensional projection of $B_\infty^n$ verifies the variance conjecture.

In this paper we will prove the following
\begin{thm}\label{TheoremVariance}
There exists an absolute constant $C$ such that for any $1\leq k\leq \sqrt{n}$ and any $E\in G_{n,n-k}$, if $\mu$ denotes the uniform probability measure on $K=P_E B_\infty^n$, then
$$
\textrm{Var}_\mu|x|^2\leq C \lambda_\mu^2\E_\mu|x|^2.
$$
\end{thm}
We will also prove the following theorem, which shows that for $k$ in a larger range, the variance conjecture is true for the family of uniform probabilities on the projections of $B_\infty^n$ on a random $(n-k)$-dimensional subspace. For that matter, we denote by $\mu_{n,n-k}$ the Haar probability measure on $G_{n,n-k}$.
\begin{thm}\label{TheoremVarianceRandom}
There exist absolute constants $C,c_1,c_2$ such that for any $1\leq k\leq\frac{n^\frac{2}{3}}{(\log n)^\frac{1}{3}}$, if $\mu$ denotes the uniform probability measure on $K=P_E(B_\infty^n)$, the measure $\mu_{n,n-k}$ of the subspaces $E\in G_{n,n-k}$ for which
$$
\textrm{Var}_\mu|x|^2\leq C\lambda_\mu^2\E_\mu|x|^2
$$
is greater than $1-c_1e^{-c_2n^\frac{2}{3}(\log n)^\frac{2}{3}}$.
\end{thm}

The main tool to prove both theorems will be to decompose an integral on $K$ as the sum of the integrals on the projections of some $(n-k)$-dimensional faces. It was proved in \cite{ABBW} that for any $E\in G_{n,n-k}$ there exist $F_1,\dots, F_l$ a set of $(n-k)$-dimensional faces of $B_\infty^n$ such that for any integrable function on $K$
\begin{eqnarray}\label{CauchyLowerDimension}
\E_\mu f&=&\frac{1}{|K|}\int_Kf(x)dx=\sum_{i=1}^l\frac{|P_E(F_i)|}{|K|}\E_{P_E(F_i)}f(x)\cr
&=&\sum_{i=1}^l\frac{|P_E(F_i)|}{|K|}\frac{1}{|P_E(F_i)|}\int_{P_E(F_i)}f(x)dx\cr
&=&\sum_{i=1}^l\frac{|P_E(F_i)|}{|K|}\frac{1}{|F_i|}\int_{F_i}f(P_Ex)dx\cr
&=&\sum_{i=1}^l\frac{|P_E(F_i)|}{|K|}\E_{F_i}f(P_Ex),
\end{eqnarray}
where we have denoted by $|\cdot|$ the relative volume of a convex body to the affine subspace in which it lies, $\E_{F_i}$ and by $\E_{P_E(F_i)}$ the expectation with respect to the uniform probability on the face $F_i$ and on its projection $P_E(F_i)$. In particular
$$
\sum_{i=1}^l\frac{|P_E(F_i)|}{|K|}=1.
$$
Notice that the $(n-k)$-dimensional faces of $B_\infty^n$ are the sets of the form
$$
F_{(i_1,\varepsilon_1, \dots, i_k, \varepsilon_k)}=\{x\in B_\infty^n\,:\, x_{i_j}=\varepsilon_j,j=1,\dots, k\},
$$
where $1\leq i_j\leq n$, $i_{j_1}\neq i_{j_2}$ and $\varepsilon_j=\pm1$.\medskip

For any $E\in G_{n,n-k}$ we write $S_E=S^{n-1}\cap E$ and denote by $\sigma_E$ the Haar probability measure on $S_E$.

\section{The variance conjecture on $(n-k)$-dimensional projections of the cube}

In this section we shall prove Theorem \ref{TheoremVariance}. 
We start with the following lemma, which can be proved by direct computation:

\begin{lemma}
Let $E\in G_{n,n-k}$, $\mu$ the uniform probability on $K=P_E(B_\infty^n)$ and $\{F_i\}_{i=1}^l$ the set of $(n-k)$-dimensional faces described in \eqref{CauchyLowerDimension}. Then
\begin{equation}\label{VarTwoTerms}
\textrm{Var}_\mu|x|^2=\sum_{i=1}^l\frac{|P_E(F_i)|}{|K|}\textrm{Var}_{P_E(F_i)}|x|^2+\sum_{i=1}^l\frac{|P_E(F_i)|}{|K|}\left(\E_{P_E(F_i)}|x|^2-\E_\mu|x|^2\right)^2.
\end{equation}

\end{lemma}
We will estimate the two summands appearing in  \eqref{VarTwoTerms}. The following lemma provides upper and lower bounds to some of parameters involved.

\begin{lemma}\label{EstimatesE|x|^2n-kFace}
Let $E\in G_{n,n-k}$. Then, for any $\theta\in S_E$ and  any $(n-k)$-dimensional face $F=F_{(i_1,\varepsilon_{1},\dots,i_k,\varepsilon_k)}$ of $B_\infty^n$  we have,
$$
\E_{P_E(F)}\langle x,\theta\rangle^2=\E_F\langle P_E x,\theta\rangle^2=\frac{1}{3}+\left(\sum_{j=1}^k\varepsilon_j\theta_{i_j}\right)^2-\frac{1}{3}\sum_{j=1}^k\theta_{i_j}^2
$$
and
$$
\E_{P_E(F)}|x|^2=\E_F|P_Ex|^2=\frac{n-k}{3}+\left|P_E\left(\sum_{j=1}^k\varepsilon_je_{i_j}\right)\right|^2-\frac{1}{3}\sum_{j=1}^k|P_E (e_{i_j})|^2.
$$
Consequently,
$$
\frac{n-2k}{3}\leq\E_{P_E(F)}|x|^2\leq\frac{n+2k}{3}, \hspace{1cm}
\frac{n-2k}{3}\leq\E_\mu|x|^2\leq\frac{n+2k}{3},
$$
and
$$
\lambda_\mu^2\geq \frac{n-2k}{3(n-k)}.
$$
\end{lemma}
\begin{proof}
For every $\theta\in S_E$, straightforward computations yield 
\begin{eqnarray*}
\E_F\langle P_E x,\theta\rangle^2&=&\E_F\langle x,\theta\rangle^2\cr
&=&\frac{1}{|B_\infty^{n-k}|}\int_{B_\infty^{n-k}}\left(\sum_{j=1}^k\varepsilon_j\theta_{i_j}+\sum_{j\not\in\{i_1,\dots,i_k\}}x_j\theta_j\right)^2dx\cr
&=&\left(\sum_{j=1}^k\varepsilon_j\theta_{i_j}\right)^2+\frac{1}{3}\left(\sum_{j\not\in\{i_1,\dots,i_k\}}\theta_j^2\right)\cr
&=&\left(\sum_{j=1}^k\varepsilon_j\theta_{i_j}\right)^2+\frac{1}{3}\left(1-\sum_{j=1}^k\theta_{i_j}^2\right).\cr
\end{eqnarray*}
This proves the first identity. Now, by integrating on $\theta\in S_E$ with respect to the uniform probability mealonsodsure and using Fubini's theorem, we obtain
\begin{eqnarray*}
\frac{1}{n-k}\E_F|P_Ex|^2dx&=&\frac{1}{3}+\frac{1}{n-k}\left|P_E\left(\sum_{j=1}^k\varepsilon_je_{i_j}\right)\right|^2-\frac{1}{3(n-k)}\sum_{j=1}^k|P_E (e_{i_j})|^2,
\end{eqnarray*}
which proves the second identity.

The bounds
$$
0\leq\left|P_E\left(\sum_{j=1}^k\varepsilon_je_{i_j}\right)\right|^2\leq\left|\sum_{j=1}^k\varepsilon_je_{i_j}\right|^2=k
$$
and
$$
0\leq |P_E (e_{i_j})|^2\leq|e_{i_j}|^2=1,
$$
prove the upper and lower bound for $\E_F|P_Ex|^2$ and by using formula \eqref{CauchyLowerDimension} we deduce the estimates for $\E_\mu|x|^2$. Finally, notice that
$$
\lambda_\mu^2\geq\int_{S_E}\E_\mu\langle x,\theta\rangle^2\ d\sigma_E(\theta)=\frac{1}{n-k}\E_\mu|x|^2\geq\frac{n-2k}{3(n-k)},
$$
which proves the last inequality.
\end{proof}

We now focus on the first summand in \eqref{VarTwoTerms}. We take into account the fact that for any $(n-k)$-dimensional face  $F=F_{(i_1,\varepsilon_1, \dots, i_k, \varepsilon_k)}$ we can write $$P_E(F)=a_F+T_F(B_\infty^{n-k})$$ where $a_F=P_E\left(\sum_{j=1}^k\varepsilon_j e_{i_j}\right)$ and $T_F\,:\,\R^{n-k}\to E$ is a linear map.

The effect of the translation map in our problem is the content of the next

\begin{lemma}
Let $\nu$ be a symmetric measure in $\R^n$,  $a\in\R^n$ and $\nu_a$ the translate measure $\nu_a(A):=\nu(A-a)$ (or equivalently, $\int_{\R^n}f(x)\ d\nu_a(x)=\int_{\R^n}f(x+a)\ d\nu(x)$  for any non negative (measurable) function $f$). Then,
$$\textrm{Var}_{\nu_a}|x|^2=\textrm{Var}_{\nu}|x|^2+4\E_\nu\langle a,x\rangle^2.$$

\end{lemma}
\begin{proof}
$$|x+a|^2=|x|^2+2\langle a,x\rangle+|a|^2$$ and
$$|x+a|^4=|x|^4+4|x|^2\langle a,x\rangle+4|a|^2\langle a,x\rangle+2|x|^2|a|^2+4\langle a,x\rangle^2+|a|^4$$

Taking expectations and using symmetry we have,

$$\textrm{Var}_{\nu_a}|x|^2=\E_{\nu_a}|x|^4-\big(\E_{\nu_a}|x|^2\big)^2=
\E_{\nu}|x+a|^4-\big(\E_{\nu}|x+a|^2\big)^2=$$
$$=\E_{\nu}|x+a|^4-\big(\E_{\nu}|x|^2+|a|^2\big)^2=
\textrm{Var}_{\nu}|x|^2+4\E_\nu\langle a,x\rangle^2$$
\end{proof}

Taking in the previous lemma $\nu$ as the uniform probability measure on $T_F(B_\infty^{n-k})$ we have

\begin{cor}\label{VarianceTraslation}
Let $F=F_{(i_1,\varepsilon_1, \dots, i_k, \varepsilon_k)}$ be an $(n-k)$-dimensional face of $B_\infty^n$ and let $P_E(F)=a_F+T_F(B_\infty^{n-k})$ as above. Then
$$\textrm{Var}_{P_E(F)}|x|^2=\textrm{Var}_{T_F(B_\infty^{n-k})}|x|^2+4\E_{T_F(B_\infty^{n-k})}\langle a_F,x\rangle^2$$
\end{cor}

\begin{lemma}\label{EstimatesLinearTransformations}
Let $E\in G_{n,n-k}$. Then, for any $\theta\in S_E$ and  any $(n-k)$-dimensional face of $B_\infty^n$, $F=F_{(i_1,\varepsilon_{1},\dots,i_k,\varepsilon_k)}$, if  $P_E(F)=a_F+T_F(B_\infty^{n-k})$ as above, we have
\begin{eqnarray*}
\E_{T_F(B_\infty^{n-k})}\langle x,\theta\rangle^2=\frac{1}{3}-\frac{1}{3}\sum_{j=1}^k\theta_{i_j}^2\quad \big(\leq\frac{1}{3}\big)
\end{eqnarray*}
and
$$
\E_{T_F(B_\infty^{n-k})}|x|^2=\frac{n-k}{3}-\frac{1}{3}\sum_{j=1}^k|P_E(e_{i_j})|^2\quad\big(\leq\frac{n-k}{3}\big)
$$
\end{lemma}

\begin{proof}
Notice that
\begin{eqnarray*}
\E_{P_E(F)}\langle x, \theta\rangle^2&=&\E_{T_F(B_\infty^n)}\langle a_F+x,\theta\rangle^2=\langle a_F,\theta\rangle^2+\E_{T_F(B_\infty^n)}\langle x,\theta\rangle^2\cr
&=&\left(\sum_{j=1}^k\varepsilon_j\theta_{i_j}\right)^2+\E_{T_F(B_\infty^n)}\langle x,\theta\rangle^2.\cr
\end{eqnarray*}

On the other hand, by Lemma \ref{EstimatesE|x|^2n-kFace}

\begin{eqnarray*}
\E_{P_E(F)}\langle x, \theta\rangle^2=\E_F\langle P_E x,\theta\rangle^2=\frac{1}{3}+\left(\sum_{j=1}^k\varepsilon_j\theta_{i_j}\right)^2-\frac{1}{3}\sum_{j=1}^k\theta_{i_j}^2,
\end{eqnarray*}
and we obtain the result. By integrating in $\theta\in S_E$ with respect to the uniform measure and using Fubini's theorem we obtain the second identity.
\end{proof}

As a consequence we have the following lemma, which gives an upper bound for the first term in \eqref{VarTwoTerms} of the right order for the variance conjecture to be true as long as $k\leq \frac{n}{3}$.

\begin{lemma}\label{BoundFirstterm}
Let $E\in G_{n,n-k}$. Then, for  any $(n-k)$-dimensional face $F$ of $B_\infty^n$ we have,
$$
\textrm{Var}_{P_E(F)}|x|^2\leq Cn.
$$
Consequently, there exists an absolute constant $C$ such that if $k\leq\frac{n}{3}$ and $\{F_i\}_{i=1}^l$ is the set of $(n-k)$-dimensional faces described in \eqref{CauchyLowerDimension} then
$$
\sum_{i=1}^l\frac{|P_E(F_i)|}{|K|}\textrm{Var}_{P_E(F_i)}|x|^2\leq C\lambda_\mu^2\E_\mu|x|^2.
$$
\end{lemma}
\begin{proof}
By Corollary \ref{VarianceTraslation}, we have that for any such $F$
$$
\textrm{Var}_F|P_Ex|^2=\textrm{Var}_{T_F(B_\infty^{n-k})}|x|^2+4\E_{T_F(B_\infty^{n-k})}\langle a_F,x\rangle^2.
$$
Since $B_\infty^{n-k}$ verifies the Kannan-Lov\'asz-Simonovits conjecture, every linear transform of it verifies the variance conjecture and therefore there exists an absolute constant $C$ such that
$$
\textrm{Var}_{T_F(B_\infty^{n-k})}|x|^2\leq C\lambda_{T_F(B_\infty^{n-k})}^2\E_{T_F(B_\infty^{n-k})}|x|^2
$$
Since by Lemma \ref{EstimatesLinearTransformations} the two factors involved are bounded by $\frac13$ and $\frac{n-k}3$ respectively, we have $
\textrm{Var}_{T_F(B_\infty^{n-k})}|x|^2\le C(n-k)$.

On the other hand, by Lemma \ref{EstimatesLinearTransformations},

$$\E_{T_F(B_\infty^{n-k})}\langle a_F,x\rangle^2\leq \frac13|a_F|^2=\frac13\left|P_E\left(\sum_{j=1}^k\varepsilon_j e_{i_j}\right)\right|^2
\leq\frac13\left|\left(\sum_{j=1}^k\varepsilon_j e_{i_j}\right)\right|^2=\frac13k.$$

 Therefore, there exists an absolute constant $C$ such that
$$
\textrm{Var}_{P_E(F)}|x|^2\leq C(n-k+k)=Cn,
$$
which proves the first part of the Lemma.

For the second part,  notice that by Lemma \ref{EstimatesE|x|^2n-kFace} we have
$$
\lambda_\mu^2\E_\mu|x|^2\geq \frac{({n-2k})^2}{9(n-k)}\geq\frac{n}{54}
$$
when $1\leq k\leq\frac{n}{3}$ and now the second part of the Lemma easily follows.
\end{proof}

For the second summand of \eqref{VarTwoTerms} we invoke once again Lemma \ref{EstimatesE|x|^2n-kFace}. The estimates therein provide an upper bound of the right order for the variance conjecture to hold as long as $k\leq\sqrt{n}$.

\begin{lemma}\label{BoundSecondTermVariance}
Let $E\in G_{n,n-k}$ and let $\mu$ be the uniform probability on $K=P_E(B_\infty^n)$. Then for any $(n-k)$-dimensional face $F$ of $B_\infty^n$ we have,
$$
\left|\E_{F}|P_Ex|^2-\E_\mu|x|^2\right|\leq \frac{4k}{3}.
$$
Consequently, there exists an absolute constant $C$ such that if $k\leq\sqrt n$ and $\{F_i\}_{i=1}^l$ is the set of $(n-k)$-dimensional faces described in \eqref{CauchyLowerDimension},
$$
\sum_{i=1}^l\frac{|P_E(F_i)|}{|K|}\left(\E_{F_i}|P_Ex|^2-\E_\mu|x|^2\right)^2\leq C\lambda_\mu^2\E_\mu|x|^2.
$$
\end{lemma}
\begin{proof}
By Lemma \ref{EstimatesE|x|^2n-kFace} we have that for any $(n-k)$ dimensional face $F$
$$
-\frac{4k}{3}\leq \E_{F}|P_Ex|^2-\E_\mu|x|^2\leq \frac{4k}{3}.
$$
Therefore,
$$
\sum_{i=1}^l\frac{|P_E(F_i)|}{|K|}\left(\E_{F_i}|P_Ex|^2-\E_\mu|x|^2\right)^2\leq \frac{16k^2}{9}.
$$
On the other hand, using as above the bound $
\lambda_\mu^2\E_\mu|x|^2\geq\frac{n}{54}
$
we have that if $k\leq\sqrt{n}$
$$
\sum_{i=1}^l\frac{|P_E(F_i)|}{|K|}\left(\E_{F_i}|P_Ex|^2-\E_\mu|x|^2\right)^2\leq C\lambda_\mu^2\E_\mu|x|^2.
$$
\end{proof}

Lemmas \ref{BoundFirstterm} and \ref{BoundSecondTermVariance}, together with formula \eqref{VarTwoTerms} prove Theorem \ref{TheoremVariance}.

\section{The variance conjecture on random $(n-k)$-dimensional projections of the cube}

We will show that we can improve the range of the codimension $k$ for which the variance conjecture remains true on a random subspace $E\in G_{n,n-k}$. In order to do that we will consider, for any $(n-k)$-dimensional face $F$ of $B_\infty^n$, the function $f: G_{n,n-k}\to\R$ given by $f(E)=\E_F|P_E x|^2$ and make use of the concentration of measure theorem, proved by Gromov and Milman,  on $G_{n,n-k}$ (see, for instance, \cite{MS}). $O(n)$ denotes the orthogonal group equipped with the Hilbert-Schmidt distance $\Vert\cdot\Vert_{HS}$ and we represent any $U\in O(n)$ by $U=(u_1,\dots, u_n)$, where $(u_i)$ is an orthonormal basis of $\R^n$. 

\begin{thm}[Concentration of measure]\label{ConcentrationOfMeasure}
Let $f\,:\, G_{n,n-k}\to\R$ be a Lipschitz function  with Lipschitz constant $\sigma$ with respect to the distance
\begin{align*}
d(E_1,E_2)=\inf&\left\{\Vert U-V\Vert_{HS}: U,V\in O(n),\,E_1=\textrm{span}\{u_1,\dots, u_{n-k}\}\right.,\cr
&\left.E_2=\textrm{span}\{v_1,\dots, v_{n-k}\}\right\}.
\end{align*}
Then, for every $\lambda>0$
$$
\mu_{n,n-k}\left\{E\in G_{n,n-k}\,:\,\left|f(E)-\E f(E)|>\lambda\right|\right\}\leq c_1e^{-\frac{c_2\lambda^2n}{\sigma^2}},
$$
where $c_1$ and $c_2$ are positive absolute constants.
\end{thm}

In the following lemma we compute the expected value of $f$. Let us point out that what matters to us for our purposes is that, due to the symmetries of $B_\infty^n$, its value does not depend on the face $F$. Nevertheless, we compute its exact value.

\begin{lemma}
Let $F$ be an $(n-k)$-dimensional face of $B_\infty^n$. Then
$$
\int_{G_{n,n-k}}\E_F|P_E x|^2d\mu(E)=\frac{(n-k)(n+2k)}{3n}.
$$
\end{lemma}

\begin{proof}
Notice that, by Fubini's theorem and the uniqueness of the Haar measure $\sigma$ on $S^{n-1}$ we have
\begin{eqnarray*}
\int_{G_{n,n-k}}\!\!\!\!\!\E_F|P_E x|^2d\mu_{n,n-k}(E)&=&(n-k)\int_{G_{n,n-k}}\!\!\!\!\!\E_F\int_{S_E}\!\langle P_E x,\theta\rangle^2d\sigma_E(\theta)d\mu_{n,n-k}(E)\cr
&=&(n-k)\int_{G_{n,n-k}}\!\!\E_F\int_{S_E}\langle x,\theta\rangle^2d\sigma_E(\theta)d\mu_{n,n-k}(E)\cr
&=&(n-k)\E_F\int_{G_{n,n-k}}\int_{S_E}\langle x,\theta\rangle^2d\sigma_E(\theta)d\mu_{n,n-k}(E)\cr
&=&(n-k)\E_F\int_{S^{n-1}}\langle x,\theta\rangle^2d\sigma(\theta)\cr
&=&\frac{n-k}{n}\E_F|x|^2\cr
&=&\frac{n-k}{n}\left(k+\E_{B_\infty^{n-k}}|x|^2\right)\cr
&=&\frac{n-k}{n}\left(k+\frac{n-k}{3}\right)\cr
&=&\frac{n-k}{n}\frac{n+2k}{3}.\cr
\end{eqnarray*}
\end{proof}

In the following lemma we estimate the Lipschitz constant of $f$ with respect to the distance defined in Theorem \ref{ConcentrationOfMeasure}. Notice that, as before, its value does not depend on $F$.
\begin{lemma}
Let $F=F_{(i_1,\varepsilon_{1},\dots,i_k,\varepsilon_k)}$ be an $(n-k)$-dimensional face of $B_\infty^n$ and let $f\,:\, G_{n,n-k}\to\R$ be the function $f(E)=\E_F|P_E x|^2$. For any $E_1,E_2\in G_{n,n-k}$ we have
$$
\left|f(E_1)-f(E_2)\right|\leq \frac{8\sqrt{2}k}{3}d(E_1,E_2).
$$
\end{lemma}

\begin{proof}
Let $E_1,E_2\in G_{n,n-k}$. By Lemma \ref{EstimatesE|x|^2n-kFace} we have
\begin{eqnarray*}
&&\left|f(E_1)-f(E_2)\right|=\left|\E_F|P_{E_1}x|^2-\E_F|P_{E_2} x|^2\right|\cr
&=&\left|\left|P_{E_1}\!\left(\sum_{j=1}^k\varepsilon_j e_{i_j}\!\right)\right|^2\!-\left|P_{E_2}\!\left(\sum_{j=1}^k\varepsilon_j e_{i_j}\!\right)\right|^2\!-\frac{1}{3}\left(\!\sum_{j=1}^k|P_{E_1}(e_{i_j})|^2-|P_{E_2}(e_{i_j})|^2\!\right)\right|\cr
&\leq&\left|\left|P_{E_1}\left(\sum_{j=1}^k\varepsilon_j e_{i_j}\right)\right|^2-\left|P_{E_2}\left(\sum_{j=1}^k\varepsilon_j e_{i_j}\right)\right|^2\right|\cr
&+&\frac{1}{3}\left|\left(\!\sum_{j=1}^k|P_{E_1}(e_{i_j})|^2-|P_{E_2}(e_{i_j})|^2\right)\right|\cr
&=&\left|\left|P_{E_1}\!\left(\!\sum_{j=1}^k\varepsilon_j e_{i_j}\!\right)\right|\!+\!\left|P_{E_2}\!\left(\!\sum_{j=1}^k\varepsilon_j e_{i_j}\!\right)\right|\right|\left|\left|P_{E_1}\!\left(\sum_{j=1}^k\varepsilon_j e_{i_j}\!\right)\right|\!-\!\left|P_{E_2}\!\left(\!\sum_{j=1}^k\varepsilon_j e_{i_j}\!\right)\right|\right|\cr
&+&\frac{1}{3}\left(\sum_{j=1}^k\left||P_{E_1}(e_{i_j})|+|P_{E_2}(e_{i_j})|\right|\left||P_{E_1}(e_{i_j})|-|P_{E_2}(e_{i_j})|\right|\right)\cr
&\leq&2\sqrt{k}\left|\left(P_{E_1}-P_{E_2}\right)\left(\sum_{j=1}^k\varepsilon_j e_{i_j}\right)\right|\cr
&+&\frac{2}{3}\sum_{j=1}^k|\left(P_{E_1}-P_{E_2}\right)(e_{i_j})|\cr
&\leq&2k\Vert P_{E_1}-P_{E_2}\Vert +\frac{2k}{3}\Vert P_{E_1}-P_{E_2}\Vert\cr
&=&\frac{8k}{3}\Vert P_{E_1}-P_{E_2}\Vert.\cr
\end{eqnarray*}
Notice that for any $U,V\in O(n)$ such that $E_1=\textrm{span}\{u_1,\dots, u_{n-k}\}$ and $E_2=\textrm{span}\{v_1,\dots, v_{n-k}\}$ we can write $P_{E_1}=\sum_{j=1}^{n-k} u_j\otimes u_j$ and $P_{E_2}=\sum_{j=1}^{n-k} v_j\otimes v_j$, and then, for any such $U,V$
\begin{eqnarray*}
\Vert P_{E_1}-P_{E_2}\Vert^2&\leq& \Vert P_{E_1}-P_{E_2}\Vert_{HS}^2=2(n-k)-2\sum_{i,j=1}^{n-k}\langle u_i,v_j\rangle^2\cr
&\leq&2\sum_{j=1}^{n-k}(1-\langle u_j,v_j\rangle^2)\leq2\sum_{j=1}^{n-k}|u_j-v_j|^2\leq2\sum_{j=1}^{n}|u_j-v_j|^2\cr
&=&2\Vert U-V\Vert_{HS}^2,
\end{eqnarray*}
since $1-\langle u_j, v_j\rangle^2\leq 2(1-\langle u_j, v_j\rangle)=|u_j-v_j|^2$. Consequently $\Vert P_{E_1}-P_{E_2}\Vert\leq\sqrt2d(E_1,E_2)$ and we obtain the result.
\end{proof}

\begin{lemma}\label{HighProbabilityEveryFace}
Let $1\leq k\leq\frac{n^\frac{2}{3}}{(\log n)^\frac{1}{3}}$. There exist positive absolute constants $C, c_1,c_2$ such that the set
$$
\left\{E\in G_{n,n-k}\,:\,\left|\E_F|P_E x|^2-\frac{(n-k)(n+2k)}{3n}\right|>C\sqrt n\,,\,\textrm{for some }F\right\}
$$
has measure $\mu_{n,n-k}$ smaller than $c_1e^{-c_2n^\frac{2}{3}(\log n)^\frac{2}{3}}$.
\end{lemma}
\begin{proof}
Let $F$ be a fixed $(n-k)$-dimensional face of $B_\infty^n$. Then, taking $\lambda=C\sqrt{n}$ we obtain, using Theorem \ref{ConcentrationOfMeasure} that
$$
\mu_{n,n-k}\left\{E\in G_{n,n-k}\,:\,\left|\E_F|P_E x|^2-\frac{(n-k)(n+2k)}{3n}\right|> C\sqrt n\right\}\leq c_1e^{-\frac{c_2C^2 n^2}{k^2}}.
$$
Since the number of $(n-k)$-dimensional faces of $B_\infty^n$ equals $2^k\binom{n}{k}$, using the union bound we have that for any $C>0$
\begin{eqnarray*}
&&\mu_{n,n-k}\left\{E\in G_{n,n-k}\,:\,\left|\E_F|P_E x|^2-\frac{(n-k)(n+2k)}{3n}\right|>C\sqrt n\,,\,\textrm{for some }F\right\}\cr
&\leq& c_1e^{-\frac{c_2C^2 n^2}{k^2}+k\log 2+k\log\frac{en}{k}}\leq c_1e^{-\frac{c_2C^2 n^2}{k^2}+c_3k\log n}\cr
&=&c_1e^{-\frac{c_2C^2 n^2-c_3k^3\log n}{k^2}}\leq  c_1e^{-c_4n^\frac{2}{3}(\log n)^\frac{2}{3}}\cr
\end{eqnarray*}
if we choose $C$ a constant big enough.
\end{proof}

As a consequence, we obtain the following lemma, which gives an estimate of the right order for most subspaces, for the second term in \eqref{VarTwoTerms}.

\begin{lemma}\label{BoundSecondTermVariance2}
There exists an absolute constant $C$ such that for any $1\leq k\leq\frac{n^\frac{2}{3}}{(\log n)^\frac{1}{3}}$, if $\mu$ denotes the uniform probability measure on $K=P_E(B_\infty^n)$, the measure $\mu_{n,n-k}$ of the subspaces $E\in G_{n,n-k}$ for which
$$
\left|\E_{P_E(F)}|x|^2-\E_\mu|x|^2\right|\leq C\sqrt{n}
$$
for every $(n-k)$-dimensional face $F$ of $B_\infty^n$ is greater than $1-c_1e^{-Cn^\frac{2}{3}(\log n)^\frac{2}{3}}$.
\end{lemma}
\begin{proof}
By Lemma \ref{HighProbabilityEveryFace} there exists a set of subspaces $\mu_{n,n-k}$ measure greater than $1-c_1e^{-Cn^\frac{2}{3}(\log n)^\frac{2}{3}}$ such that for every $(n-k)$-dimensional face $F$ of $B_\infty^n$,
$$
\left|\E_{P_E(F)}|x|^2-\frac{(n-k)(n+2k)}{3n}\right|=\left|\E_F|P_E x|^2-\frac{(n-k)(n+2k)}{3n}\right|\leq\alpha\sqrt n
$$
and then for every $F_1, F_2$, $(n-k)$-dimensional faces
$$
\left|\E_{P_E(F_1)}|x|^2-\E_{P_E(F_2)}|x|^2\right|\leq 2\alpha\sqrt{n}.
$$
Consequently, since $\displaystyle{\E_\mu|x|^2=\sum_{i=1}^l\frac{|P_E(F_i)|}{|K|}\E_{P_E{F_i}}|x|^2}$, we have that for every  $E$ in this set and every face $F$
$$
\E_{P_E(F)}|x|^2-\E_\mu|x|^2\leq\E_{P_E(F)}|x|^2-\min_{i=1,\dots,l}\E_{P_E(F_i)}|x|^2\leq 2\alpha\sqrt{n}
$$
and
$$
\E_{F}|P_Ex|^2-\E_\mu|x|^2\geq\E_{F}|P_Ex|^2-\max_{i=1,\dots,l}\E_{F_i}|x|^2\geq -2\alpha\sqrt{n}.
$$
\end{proof}
Now we are able to prove Theorem \ref{TheoremVarianceRandom}.

\begin{proof}[Proof of Theorem \ref{TheoremVarianceRandom}]
If $1\leq k\leq\frac{n^\frac{2}{3}}{(\log n)^\frac{1}{3}}$ by Lemma \ref{EstimatesE|x|^2n-kFace} we have  $\lambda_\mu^2\E_\mu|x|^2\geq Cn$. By equation  \eqref{VarTwoTerms}, if $\{F_i\}_{i=1}^l$ are the $(n-k)$-dimensional faces of $B_\infty^n$ described in $(1)$ we have,
$$
\textrm{Var}_\mu|x|^2\leq\max_{i=1\dots,l}\textrm{Var}_{P_E(F_i)}|x|^2+\max_{i=1,\dots,l}\left(\E_{P_E(F_i)}|x|^2-\E_\mu|x|^2\right)^2.
$$
By Lemma \ref{BoundFirstterm} the first maximum is bounded from above by $Cn$ and by Lemma \ref{BoundSecondTermVariance2} there exists a set of $(n-k)$-dimensional subspaces with measure larger than  $1-c_1e^{-Cn^\frac{2}{3}(\log n)^\frac{2}{3}}$ such that the second maximum is bounded from above by $Cn$, which proves Theorem \ref{TheoremVarianceRandom}.
\end{proof}

\end{document}